\documentclass[11pt]{article}
\usepackage{amsmath,amsthm,mathrsfs}
\usepackage[english]{babel}

\usepackage{amsrefs}

\usepackage{amsfonts,amstext,amssymb,verbatim,epsfig}%,psfrag}
\usepackage{pgf,tikz}
\usepackage{float}
\usetikzlibrary{arrows}
\usepackage{hyperref}
\usepackage{appendix}
\usepackage{subfigure}
\usepackage{prettyref}
\usepackage{esint}

\usepackage{bbm,mathrsfs}

\usepackage{xcolor}
\usepackage[normalem]{ulem}

\hypersetup{colorlinks=true,pdfborder=000,  citecolor  = blue, citebordercolor= {magenta}}
\hypersetup{linkcolor=purple}
\makeatletter
\let\reftagform@=\tagform@
\def\tagform@#1{\maketag@@@{(\ignorespaces\textcolor{purple}{#1}\unskip\@@italiccorr)}}
\renewcommand{\eqref}[1]{\textup{\reftagform@{\ref{#1}}}}
\makeatother

\makeatletter
\DeclareUrlCommand\ULurl@@{%
  \def\UrlLeft{\uline\bgroup}%
  \def\UrlRight{\egroup}}
\def\ULurl@#1{\hyper@linkurl{\ULurl@@{#1}}{#1}}
\DeclareRobustCommand*\ULurl{\hyper@normalise\ULurl@}
\makeatother

\def\lessim{\ \lower4pt\hbox{$
		\buildrel{\displaystyle <}\over\sim$}\ }
\def\gessim{\ \lower4pt\hbox{$\buildrel{\displaystyle >}
		\over\sim$}\ }

\def\si{\sigma}

\def\eps{{\varepsilon}}

\newcommand{\e}{\mathbb{E}}
\newcommand{\p}{\mathbb{P}}

\newcommand{\indi}{\ensuremath{\boldsymbol 1}}

\newcommand{\kx}{\mathcal{K}}
\newcommand{\qx}{\mathcal{Q}}

\newtheorem{lemma}{\bf Lemma}

\newtheorem{theorem}[lemma]{\bf Theorem}
\newtheorem{corollary}[lemma]{\bf Corollary}

\newtheorem{example}{\bf Example}

\newtheorem{proposition}[lemma]{\bf Proposition}
\newtheorem{question}{\bf Question}

\theoremstyle{remark}

\newtheorem{remark}{Remark}

\newcommand{\8}{\infty}

\newcommand{\rz}{\mathbb{R}}
\newcommand{\px}{\mathcal{P}}

\newcommand{\de}{\delta}

\newcommand{\ux}{\mathcal{U}}
\newcommand{\ga}{\gamma}
\newcommand{\la}{\lambda}

\newenvironment{Proof of lemma}{\noindent{\bf Proof of Lemma}}{\hfill$\Box$\newline}
\newenvironment{Proof of theorem}{\noindent{\bf Proof of Theorem}}{\hfill{\footnotesize${\square}$}\newline}
\newenvironment{Proof of theorems}{\noindent{\bf Proof of Theorems}}{\hfill$\Box$\newline}
\newenvironment{Proof of proposition}{\noindent{\bf Proof of Proposition}}{\hfill$\Box$\newline}
\newenvironment{Proof of propositions}{\noindent{\bf Proof of Propositions}}{\hfill$\Box$\newline}
\newenvironment{Proof of exercise}{\noindent{\it Proof of Exercise:}}{\hfill$\Box$}
\begin{document}
\title{Existence of two-step replica symmetry breaking for the spherical mixed $p$-spin glass at zero temperature}
\author{Antonio Auffinger \thanks{Department of Mathematics, Northwestern University, tuca@northwestern.edu, research partially supported by NSF Grant CAREER DMS-1653552 and NSF Grant DMS-1517894.} \\
	\small{Northwestern University}\and Qiang Zeng \thanks{Department of Mathematics, Northwestern University, Email: qzeng.math@gmail.com.}\\
	\small{Northwestern University}}
\date{June 25, 2018}

\maketitle

\begin{abstract}
We provide the first examples of two-step replica symmetry breaking (2-RSB) models for the spherical mixed $p$-spin glass at zero temperature. Precisely, we show that for a certain class of mixtures, the Parisi measure at zero temperature is purely atomic and has exactly three distinct points in its support. We then derive a few consequences for the topology of the random landscape in these cases. Our main result also provides a negative answer to a question raised in 2011 by Auffinger and Ben Arous about the classification of pure-like and full mixture models.

\end{abstract}

\section{Introduction}\label{sect1}
For $N\geq 1$, let $S_N = \{\si\in\rz^N: \sum_{i=1}^N \si_i^2 =N\}$ be the sphere of radius $\sqrt{N}$. The Hamiltonian of the spherical mixed $p$-spin model is defined as
the centered Gaussian field indexed by $S_{N}$ with covariance
\[
\e[H_N(\si^1) H_N(\si^2)] = N\xi(R_{12})
\]
where
\[
\xi(x)=\sum_{p=2}^\8 c_p^2 x^p, \quad  c_p\ge0
\] and
$R_{12}=R_{12}(\si^{1},\si^{2})=\frac{1}{N}\sum_{i=1}^{N}\si^{1}_{i}\si^{2}_{i}$ is the normalized inner product. We assume that the variance is constant by setting
\[
\xi(1)=\sum_{p}c_p^2 =1.
\]
We also assume that $\sum_{p\geq 2} 2^{p}c_{p}<\infty$. When $\xi(x) = x^{2}$, we recover the spherical Sherrington-Kirkpatrick model \cite{KTJ76} while the choice of $\xi(x)=x^{p}$, $p\geq 3$ represents the spherical pure $p$-spin model. We say $\xi$ is a convex model if it is a convex function.

Let $\kx$ denote the collection of all measures on $[0,1]$ which are of the form
\[
\nu(ds) = \indi_{[0,1)}(s)\ga(s) ds + \Delta\de_{\{1\}}(ds),
\]
where $\ga(s)$ is a nonnegative and nondecreasing function on $[0,1)$ with right-continuity and $\Delta>0$.
Define the Crisanti-Sommers functional \cite{CS} for $\nu\in \kx$
\begin{align*}
\qx(\nu) = \frac12\Big( \int_0^1\xi'(s)\nu(ds) +\int_0^1 \frac{dq}{\nu((q,1])}\Big).
\end{align*}
The minimizer of $\qx(\nu)$ exists and is unique \cite[Theorem 1]{ArnabChen15}. We denote it by \[\nu_P(ds) = \indi_{[0,1)}(s)\ga_P(s) ds + \Delta_P\de_{\{1\}}(ds).\] The measure $\nu_P=\nu_{P}(\xi)$ on $[0,1]$ is called the Parisi measure at zero temperature. Its importance lies primarily on the fact that it describes the energy landscape of $H_{N}$ near the global minimum \cite{AC17}; for instance
\[
 \lim_{N\to \infty} \frac{1}{N} \min_{\sigma \in S_{N}} H_{N}(\sigma)= -\qx(\nu_{P}),
\]
see \cite[Theorem 1]{ArnabChen15} and \cite[Theorem 1.1.3]{JT16}.

The aim of this article is to study the structure of the support of the Parisi measure at zero temperature as a function of the model $\xi$. Our main result shows that it is possible to find functions $\xi$ such that the Parisi measure of the model is atomic with exactly three atoms in its support.\footnote{At zero temperature the Parisi measure always have an atom at $1$. In \cite{AC17}, the authors used the word Parisi measure for the measure induced by the function $\ga(s)$.}  This is referred as two levels of replica symmetry breaking (2-RSB) in mathematical physics nomenclature ($k$-RSB refers to atomic measures with $k+1$ atoms). For more information on the physics literature, we refer the reader to \cite{MPV}. Precisely, we prove the following.

\begin{theorem}\label{mainthm}
There exist models $\xi$ such that for some positive constants $m_{1}, m_{2}, \Delta_{P}$ and $q \in (0,1)$ one has
\[
\nu_{P}(\xi)=m_{1} \delta_{0} + m_{2} \delta_{q} + \Delta_{P} \delta_{1}.
\]
\end{theorem}

The classification of levels of replica symmetry breaking (RSB) has a long history in the physics community, see \cite[Chapter 3]{MPV}. Although one can artificially cook up stochastic processes with arbitrary levels of finite RSB (for instance the GREM \cite{BovierK,Derrida}), it was believed that only replica symmetric (RS), $1$-RSB and FRSB models would appear ``naturally''.\footnote{See, for instance, Bolthausen's survey article \cite[Page 15]{Bolt00}.} For instance, all models defined on the discrete hypercube would exhibit infinite levels of replica symmetry breaking at low enough temperature (FRSB), that is, the Parisi measure will have infinite many points in its support\footnote{One expects in this case that the support of the Parisi measure contains an interval.}. Moreover, it is known that the spherical pure $p$-spin model is $1$-RSB at low temperature (see \cite[Section 4]{PT}).

More recently, existence of $2$-RSB spherical models were suggested at positive temperature in the physics literature in an insightful paper by Crisanti-Leuzzi \cite{CL07} (see also \cite{CriLeuzz,Krakoviack,CL07a} for a long debate on these cases). Several rigorous results on properties of Parisi measures appear in the works of Talagrand \cite{Tal06}, Panchenko-Talagrand \cite{PT}, and more recently in \cite{ABC, AB, AC15a, AC17, ACZ17, JT16, Pan}. We refer the readers to Talagrand \cite{Tal06} and Auffinger-Chen \cite{AC14} for an introduction on Parisi measures.

As far as we know, Theorem \ref{mainthm} is the first rigorous result to provide ``natural'' examples of spin glass models beyond RS, $1$-RSB and FRSB phases. This result is somewhat surprising. Indeed, prior to the results of this paper, it was perceived that the spherical model had only one of two possible behaviors at zero temperature: the Parisi measure would either be  $1$-RSB (as in the pure $p$-spin) or FRSB. This distinction was the motivation behind the classification of pure-like and full mixture models introduced in \cite{AB}. It turns out that this is a tepid classification: we find pure-like $2$-RSB models and full mixture $2$-RSB models.
Furthermore, the examples that we construct provide a  negative answer to Question \cite[Question 4.1]{AB} that does not involve using the spherical SK model. We refer the reader to Remark \ref{rem111} in Section \ref{sec:examples} for more on these examples.

A word of comment is needed. The structure of the Parisi measure at zero temperature in the case of Ising spins (when $S_{N}$ is replaced by $\{\pm 1 \}^{N}$) was recently determined in \cite{ACZ17}. The result there is strikingly different: on $\{\pm 1\}^{N}$, all models have infinite levels of replica symmetry breaking, as predicted by physicists. Comparing with Theorem \ref{mainthm}, this shows that the sphere has a richer collection of models than the hypercube (at least in terms of different levels of RSB).

We also study the landscape of $H_{N}$ when the model is $2$-RSB. We provide information on the topology of level sets near the global minimum of $H_{N}$. In order to state our results, set, for any $\eta>0$,
\begin{equation*}
		\mathcal L(\eta):=\big \{ \sigma \in S_{N}: H_{N}(\sigma) \leq - N (\qx(\nu_{P}) - \eta)\big\}.
		\end{equation*}
For any Borel measurable set $A \subset [-1,1]$ set
\[
\mathbb P_{N}(\eta,A) := \mathbb P(\exists \; \sigma^{1}, \sigma^{2 } \in  \mathcal L(\eta), \text{ with } R_{1,2} \in A).
\]
 For any $\varepsilon>0$, also set
\[
 A_{\varepsilon} = \{ x \in [-1,1] : \exists \; y \in A \text{ with } |x-y| <\varepsilon \}.
\]
We  write $-A =\{ x \in [-1,1]: -x \in A\}.$
\begin{theorem}\label{thm:2RSB}

Let $\xi$ be a 2-RSB convex model and $S=\{0,q,1\}$ denote the support of $\nu_{P}(\xi)$. Suppose $\psi_1(a)<0, \forall a\in(0,q)$ and $\psi_2(b)<0, \forall b\in (q,1)$, where $\psi_1$ and $\psi_2$ are given in \eqref{eh1a} and \eqref{eh2b}.  Then for any $\varepsilon>0$, there exist $\eta,K>0$ such that for all $N\geq 1,$
		\begin{align*}
		\mathbb{P}_N\bigl(\eta,(S_{\varepsilon}\cup -S_{\varepsilon})^{c}\bigr)&\leq Ke^{-\frac{N}{K}}.
		\end{align*}
\end{theorem}

The conclusion of Theorem \ref{thm:2RSB} was only known in the case of $1$-RSB models \cite[Theorem 6]{AC17}. The analysis of the $2$-RSB case is more involved and it is done in Section \ref{proofofthm2}.  The conditions on $\psi_1$ and $\psi_2$ in Theorem \ref{thm:2RSB} are counterparts of the more restrictive condition $\zeta(s)<0, \forall s\in(0,1)$ for 1-RSB \cite[Theorem 6]{AC17}.  These conditions on $\psi_1$ and $\psi_2$  always hold for the models we consider in this paper; see Proposition \ref{prop2} and Remark \ref{rem3}.

Let us now describe a major consequence of Theorem  \ref{thm:2RSB}.  For any $\varepsilon,\eta,K>0$,  denote by
    \begin{align*}
    %\label{prob}
    \p_N(\varepsilon,\eta,K)
    \end{align*}
    the probability that there exists a subset $O_N\subset S_N$ such that
    \begin{enumerate}
    	\item[$(i)$] $O_N\subset \mathcal{L}(\eta).$
    	\item[$(ii)$] $O_N$ contains at least $Ke^{N/K}$ many elements.
    	\item[$(iii)$] $|R(\sigma,\sigma')|\leq\varepsilon$ for all distinct $\sigma,\sigma'\in O_N.$
    \end{enumerate}
It is known \cite[Proposition 2]{AC17} that  for any $\varepsilon,\eta>0$, there exists $K>0$ such that for any $N\geq 1,$
		\begin{align}\label{es}
		%\label{thm1.1:eq1}
		\p_N(\varepsilon,\eta,K)\geq 1-Ke^{-N/K}.
		\end{align}
If we combine \eqref{es} with Theorem \ref{thm:2RSB} we obtain the following orthogonal decomposition of local minima of $H_{N}$. Let
\[
\text{Crt}(\eta) = \{ \sigma \in S_{N}: \nabla H_{N} =0, H_{N}(\sigma) \leq -N (\qx(\nu_{P}) - \eta) \}.
\]
\begin{corollary}\label{cor:1}
There exist $\eta'>0, \epsilon' >0, K>0$ so that for all $N\geq 1$ the event defined by
\begin{enumerate}
\item[$(i)$]  $\mathbb \# \text{Crt} (\eta') \geq Ke^{N/K},$
\item[$(ii)$]  Any continuous path connecting two points in $\text{Crt}(\eta')$ must leave the level set  $\mathcal L(\eta'+\epsilon')$,
\end{enumerate}
has probability  at least $1-Ke^{-\frac{N}{K}}$.
\end{corollary}

	A major feature of Theorem \ref{thm:2RSB} and Corollary  \ref{cor:1} is that we can always find exponentially many local minima of $H_{N}$ around the ground state energy. Furthermore, if we think $H_{N}$ as a random landscape on the sphere, in order to go from one deep local minimum to another one we must climb a diverging energy barrier (in $N$). This fact was predicted to hold in $1$-RSB models and known for the pure $p$-spin \cite{ABC, AC17}. Now that such phenomenon occurs in the 2-RSB phase, it is natural to ask the following question:
	\begin{question}
	In the family of spherical mixed $p$-spin models, do there exist $k$-RSB models (at zero or positive temperature) for any $k\ge 3, k\in \mathbb{N}$?
	\end{question}
We anticipate the answer to be positive. In fact, our results for 2-RSB suggest that the energy landscape of $k$-RSB models for $k\ge 3$ should satisfy Theorem \ref{thm:2RSB} and Corollary \ref{cor:1} provided such models exist. However, since we have discovered $2$-RSB examples, the next desirable step is not only to answer the question above but to give a complete characterization  of the Parisi measure as a function of $\xi$. This seems beyond reach at this moment.

	The organization of the paper is the following. In the next section, we establish sufficient and necessary conditions for the model to be $2$-RSB. These conditions follow from a careful analysis of a criterion developed by Chen and Sen in \cite{ArnabChen15}. This criterion was also explored in the $1$-RSB case by Auffinger and Chen in \cite{AC17}. Compared to the $1$-RSB case, the complexity of the analysis of the $2$-RSB is acutely more demanding. In Section \ref{s:s+p}, we provide sufficient conditions for $s+p$ models (i.e. $\xi(x)=(1-\la) x^s +\la x^p$) to be 2-RSB. These conditions are easier to check in practice, and are indeed verified in Section \ref{sec:examples} for some models, proving Theorem~\ref{mainthm}. We study the energy landscape of 2-RSB models in Section \ref{proofofthm2}.

\subsection*{Acknowledgements} We would like to thank Wei-Kuo Chen for helpful discussions in the initial stage of this project. We also thank the referees for many suggestions, which have helped to improve the presentation of the paper.

\section{Criteria for a $2$-RSB Parisi measure}
We start this section by recalling the following criterion derived in Chen-Sen \cite{ArnabChen15}. Similar criteria also appeared in \cite{AC17, Tal06}. Recall from Section \ref{sect1} that $\nu\in \kx$ is a measure of the form
\[
\nu(ds) = \indi_{[0,1)}(s)\ga(s) ds + \Delta\de_{\{1\}}(ds).
\]

\begin{theorem}[Chen-Sen \cite{ArnabChen15}] \label{thmCS}For $\nu \in \kx$,  let
\begin{align*}
g(u) &= \int_u^1 \bigg(\xi'(s)-\int_0^s \frac{dr}{\nu((r,1])^2} \bigg)ds.
\end{align*}
Then $\nu$ is the Parisi measure for the model $\xi$ if and only if
\[\xi'(1)=\int_0^1\frac{dr}{\nu((r,1])^2},\]
 the function $g$ satisfies
\[ \min_{u\in[0,1]}g(u)\geq 0,\]
and for $S:=\{ u \in [0,1): g(u)=0\}$ one has $\rho(S)=\rho([0,1))$. Here $\rho$ is the measure induced by $\ga$, i.e. $\rho([0,s])=\ga(s)$.
\end{theorem}

Here the notation $\nu$ is different from the one used in \cite[Theorem 2]{ArnabChen15}, but it coincides with \cite{JT16, AC17}. Let $\nu(ds) = \ga(s) ds +\Delta\de_{\{1\}}(ds)$ where \begin{align}\label{2rsbga}
\ga(s)=A_1\indi_{[0,q)}(s)+A_2\indi_{[q,1)}(s),\quad 0<A_1<A_2.
\end{align}
 Note that
\[
\nu((r,1])=\begin{cases}
A_1(q-r)+A_2(1-q)+\Delta,& r\le q, \\
A_2(1-r)+\Delta,& q< r\le 1.
\end{cases}
\]
%Let
%\begin{align*}
%g(u) &= \int_u^1 (\xi'(s)-\int_0^s \frac{dr}{\nu([r,1])^2} )ds.
%\end{align*}
A direct computation yields for $0\le u\le q$
\begin{align*}
g(u)&= \xi(1)-\xi(u)-\frac1{A_1^2}\log\Big(1+\frac{A_1(q-u)}{A_2(1-q)+\Delta} \Big) + \frac{q-u}{A_1[A_1q+A_2(1-q)+\Delta]}\\
& - \frac1{A_2^2}\log\Big(1+\frac{A_2(1-q)}{\Delta}\Big) +\frac{1-q}{A_2(A_2(1-q)+\Delta)}-\frac{(1-q)q}{[A_2(1-q)+\Delta][A_1q+A_2(1-q) +\Delta]};
\end{align*}
and for $q<u\le 1$,
\begin{align*}
g(u) = \xi(1)-\xi(u)-\frac{q(1-u)}{[A_2(1-q)+\Delta][A_1q+A_2(1-q)+\Delta]}\\
+\frac{1-u}{A_2[A_2(1-q)+\Delta]}-\frac1{A_2^2}\log\Big(1+\frac{A_2(1-u)} {\Delta}\Big).
\end{align*}
By Theorem \ref{thmCS}, $\nu$ is the Parisi measure for $\xi$ if and only if
\begin{align*}
  g(0)=0,\quad g(q)=0,\quad g'(q)=0,\\
    \xi'(1)=\int_0^1\frac{dr}{\nu((r,1])^2},\quad g(u)\ge0, u\in[0,1].
\end{align*}
The condition $g'(q)=0$ follows from other conditions. We include it here because it simplifies a lot of computations in the following. Unraveling the first four conditions by using the expressions of $g(u)$ and $\nu((r,1])$, we find
\begin{align*}
  \xi(q)&=\frac1{A_1^2}\log\Big(1+\frac{A_1q}{A_2(1-q)+\Delta}\Big) -\frac{q}{A_1(A_1q+A_2(1-q)+\Delta)},\\
 \xi(1)-\xi(q)&= \frac{(1-q)q}{[A_2(1-q)+\Delta][A_1q+A_2(1-q)+\Delta]} +\frac1{A_2^2}\log\Big(1+\frac{A_2(1-q)}{\Delta}\Big)-\frac{1-q}{A_2(A_2(1-q)+\Delta)} ,\\
 \xi'(1)&= \frac{q}{[A_2(1-q)+\Delta][A_1q+A_2(1-q)+\Delta]} +\frac{1-q}{\Delta(A_2(1-q)+\Delta)},\\
 \xi'(q)&= \frac{q}{[A_2(1-q)+\Delta][A_1q+A_2(1-q)+\Delta]} .
\end{align*}
Now, let $z_1=A_1q/\Delta$, $z_2=A_2(1-q)/\Delta$. From the third and fourth equations above, we have
\begin{align}\label{e10}
  \Delta^2=\frac{1-q}{(1+z_2)(\xi'(1)-\xi'(q))} = \frac{q}{\xi'(q)(1+z_2)(1+z_1+z_2)}.
\end{align}
It follows that
\begin{align}\label{1z1z2}
  1+z_1+z_2= \frac{q[\xi'(1)-\xi'(q)]}{\xi'(q)(1-q)}.
\end{align}
From \eqref{e10} and \eqref{1z1z2}, we observe that $z_1$ and $\Delta$ are determined by $q$ and $z_2$.
Using the variables $z_1,z_2$, we can rewrite the first, second, and third equations as
\begin{align}
  \frac{q^2}{z_1^2}\log(1+\frac{z_1}{1+z_2})-\frac{q^2}{z_1(1+z_1+z_2)} &=\xi(q)\Delta^2,\label{e11}\\
  \frac{(1-q)q}{(1+z_2)(1+z_1+z_2)}+\frac{(1-q)^2}{z_2^2}\log(1+z_2) -\frac{(1-q)^2}{z_2(1+z_2)}&=[\xi(1)-\xi(q)]\Delta^2,\label{e12}\\
  \frac{q}{(1+z_2)(1+z_1+z_2)}+\frac{1-q}{1+z_2}&=\xi'(1)\Delta^2.\label{e13}
\end{align}
Equations \eqref{e10} and \eqref{e13} are not independent and one of them together with \eqref{e11} and \eqref{e12} determines $q,z_1,z_2$ which are the unknowns for the problem.
Consider
\begin{align}\label{xi1z}
\frac1{\xi'(1)}=\frac{1+z}{z^2}\log(1+z) -\frac1z.
\end{align}
This was used to determine the 1-RSB phase in \cite{AC17}. Note that \eqref{e11} degenerates to \eqref{xi1z} as $q\to 1$ and \eqref{e12} degenerates to \eqref{xi1z} as $q\to 0$. The condition $g(u)\ge0$ can be rewritten as for $0\le u\le q$
\begin{align}\label{gu1}
\frac{q^2}{z_1^2}\log(1+\frac{z_1(q-u)}{q(1+z_2)}) - \frac{q(q-u)}{z_1(1+z_1+z_2)}+\frac{(1-q)q}{(1+z_2)(1+z_1+z_2)}\\
+\frac{(1-q)^2}{z_2^2}\log(1+z_2)-\frac{(1-q)^2}{z_2(1+z_2)}\le (1-\xi(u))\Delta^2;\notag
\end{align}
and for $q\le u\le 1$,
\begin{align}\label{gu2}
\frac{q(1-u)}{(1+z_2)(1+z_1+z_2)}-\frac{(1-q)(1-u)}{z_2(1+z_2)}+\frac{(1-q)^2} {z_2^2}\log\Big(1+\frac{z_2(1-u)}{1-q}\Big)\le (1-\xi(u)) \Delta^2.
\end{align}
In what follows, we always assume $z_1$ is determined by $q$ and $z_2$ via \eqref{1z1z2} while sometimes we still write $z_1$ for convenience.
Plugging \eqref{e10} into \eqref{e11} and \eqref{e12}, and eliminating $z_1$ using \eqref{1z1z2}, we find
\begin{align}
  f_1(q,z_2):=-\frac{[q\xi'(q)-\xi(q)](1-q)(1+z_2)}{\xi'(1)-\xi'(q)}- q^2\log\frac{q[\xi'(1)-\xi'(q)]}{(1+z_2)\xi'(q)(1-q)}+ q^2-\frac{2\xi(q)q}{\xi'(q)} \notag \\  +
  \frac{\xi(q) q^2 [\xi'(1)-\xi'(q)]}{(1+z_2)\xi'(q)^2(1-q)}& =0,\label{e1}\\
  f_2(q,z_2):=(1-q)[\xi'(1)-\xi'(q)]\Big(\frac{1+z_2}{z_2^2}\log(1+z_2) -\frac1{z_2}\Big)+\xi'(q)(1-q)-1+\xi(q)&=0. \label{e2}
\end{align}
Despite their complicated appearance, these equations are our starting point to find 2-RSB models. First, we show that the zero sets given in \eqref{e1} and \eqref{e2} have some nice properties. Note that both functions $f_1$ and $f_2$ are $C^{\infty}$ in their domains.

\begin{proposition}\label{p0f1f2}
 Let $\xi(x)=\sum_{p=2}^\8 c_p^2 x^p\neq x^2$ and $q\in(0,1)$. Then for every $q \in (0,1)$ we have the following:
 \begin{enumerate}
\item[(i)] There are exactly two critical points $-1<z_2^s<z_2^b$ of $f_1(q,\cdot)$ on $(-1,+\8)$, where $z_2^s=z_2^s(q)$ is a local minimum and $z_2^b=z_2^b(q)$ is a local maximum; moreover, $f_1(q,\cdot)$ has exactly two zeros in $(-1,+\8)$, one is strictly less than $z_2^s$ and the other is $z_2^b$.
\item[(ii)] There exists a unique $z_2>0$ such that \eqref{e2} holds.
 \end{enumerate}
 \end{proposition}
 \begin{proof}
 We start by proving $(i)$. Let $w=1+z_2$ and set $\varphi(w) = f_{1}(q,w-1)$. Taking derivative in $w$, we have
\[
\varphi'(w) = -\frac{[q \xi'(q)-\xi(q)](1-q)}{\xi'(1)-\xi'(q)}+\frac{q^2}w -\frac{\xi(q)q^2[\xi'(1)-\xi'(q)]}{\xi'(q)^2(1-q)w^2}.
\]
Note that $q\xi'(q)>2\xi(q)$ for $q>0$. Solving the quadratic equation given by $\varphi'(w)=0$, we find the two roots
\[
w_{1,2} = \frac{q^2\pm [q^2-\frac{2q\xi(q)}{\xi'(q)}]}{\frac{2[q\xi'(q)-\xi(q)](1-q)}{\xi'(1)-\xi'(q)}} >0.
\]
From here we get two critical points $-1<z_2^s<z_2^b$ of $f_1(q,\cdot)$ as
\begin{align}
z_2^s(q) &= \frac{q\xi(q)[\xi'(1)-\xi'(q)]}{\xi'(q)[q\xi'(q)-\xi(q)](1-q)}-1,\label{e:z2s} \\
z_2^b(q)&=\frac{q[\xi'(1)-\xi'(q)]}{(1-q)\xi'(q)}-1.\label{e:z2b}
\end{align}
One can directly check that $f_1(q,z_2^b)=0$. Using the elementary inequality $\log(1+x)>\frac{2x}{2+x}$ for $x>0$, we have
\begin{align*}
f_1(q,z_2^s)&=-\frac{q\xi(q)}{\xi'(q)}-q^2\log\frac{q\xi'(q)-\xi(q)}{\xi(q)}+q^2-\frac{2\xi(q)q}{\xi'(q)}+\frac{q[q\xi'(q)-\xi(q)]}{\xi'(q)}\\
&=\frac{-4q\xi(q)+2q^2\xi'(q)}{\xi'(q)}-q^2\log\Big(1+\frac{q\xi'(q)-2\xi(q)}{\xi(q)}\Big)\\
&<  \frac{-4q\xi(q)+2q^2\xi'(q)}{\xi'(q)} -q^2\frac{\frac{2[q\xi'(q)-2\xi(q)]}{\xi(q)}}{2+\frac{q\xi'(q)-2\xi(q)}{\xi(q)}} =0.
\end{align*}
On the other hand, we note that $f_1(q, z_2)\to -\8$ as $z_2\to +\8$ and $f_1(q,z_2) \to +\8$ as $z_2$ decreases to $-1$. Putting what we have shown together, the function $f_1(q,\cdot)$ is strictly decreasing on $(-1,z_2^s)$ and on $(z_2^b,+\8)$, and is strictly increasing on $(z_2^s, z_2^b)$; moreover, $z_2^s$ is a local minimum and $z_2^b$ is a local maximum for $f_1(q,\cdot)$. It follows that $f_1(q,\cdot)$ has exactly two zeros in $(-1,+\8)$, one is strictly less than $z_2^s$ and the other is $z_2^b$.

We now proceed to prove $(ii)$. Note that the function $z_2\mapsto \frac{1+z_2}{z_2^2}\log(1+z_2) -\frac1{z_2}$ is a strictly decreasing function from $\frac12$ to 0 as $z_2$ goes from 0 to $+\8$. By continuity, it suffices to show that for any $q$, $f_2(q,0+)$ and $f_2(q,+\8)$ have different signs. Since $\xi'(q)= \sum_{p=2}^\8 c_p^2 p q^{p-1}$ and
  \[
  \frac{1-\xi(q)}{1-q}= \sum_{p=2}^\8 c_p^2 \sum_{k=0}^{p-1} q^{k},
  \]
  we have
  \[
  f_2(q,+\8) = \xi'(q)(1-q)-1+\xi(q) <0.
  \]
  To check $f_2(q,0+)>0$, let $h(q)=\frac12(1-q)[\xi'(1)-\xi'(q)] +\xi'(q)(1-q)-1 +\xi(q)$. Note that $h(0)=\frac12\xi'(1)-1>0$ and $h(1)=0$. Taking the derivative, we have
  \[
  h'(x)=\frac12(\xi''(x)(1-x)-[\xi'(1)-\xi'(x)]).
  \]
  But for $x\in(0,1)$
  \[
  \xi''(x) -\frac{\xi'(1)-\xi'(x)}{1-x}= \sum_{p=2}^\8 c_p^2 p\Big[(p-1) x^{p-2} - \sum_{k=0}^{p-2} x^k\Big ] <0.
  \]
  Hence, $h'(x)<0$ and $h(q)>0$.
\end{proof}

The following Lemma provides one extra condition.
\begin{lemma}
  Suppose \eqref{e10} holds, $g(q)=0$ for some $q\in(0,1)$ and $g(u)\ge0, u\in[0,1]$. Then
  \begin{align}\label{g2u0}
    \xi''(q)(1+z_2)(1-q)\le \xi'(1)-\xi'(q),\\
    \xi''(q)q(1+z_2)\le (1+z_1+z_2)\xi'(q).\notag
  \end{align}
\end{lemma}
\begin{proof}
  Since $u=q$ is a local minimum of $g(u)$, we have $g'(q)=0$ and $g''(q)\ge 0$. Recall that
  \[
  g''(u)=-\xi''(u) +\frac1{\nu([u,1])^2}.
  \]
  Then
  \[
  g''(q) = -\xi''(q)+\frac1{(A_2(1-q)+\Delta)^2}= -\xi''(q) +\frac1{(1+z_2)^2 \Delta^2}\ge0.
  \]
  The assertion follows by plugging $\Delta^2$ as given in \eqref{e10} into this equation.
\end{proof}

Suppose $z_1>0, z_2>0, q\in(0,1)$ satisfy \eqref{e1} and \eqref{e2}.  Let us define
\begin{align*}
g_1(u)  = \frac{q^2}{z_1^2}\log(1+\frac{z_1(q-u)}{q(1+z_2)}) - \frac{q(q-u)}{z_1(1+z_1+z_2)}+\frac{(1-q)q}{(1+z_2)(1+z_1+z_2)}
\\
+\frac{(1-q)^2}{z_2^2}\log(1+z_2)-\frac{(1-q)^2}{z_2(1+z_2)}- \frac{[1-\xi(u)]q}{\xi'(q)(1+z_2)(1+z_1+z_2)},\\
g_2(u)= \frac{(1-u)\xi'(q)}{(1+z_2)[\xi'(1)-\xi'(q)]}-\frac{(1-u)}{z_2(1+z_2)} +\frac{1-q} {z_2^2}\log\Big(1+\frac{z_2(1-u)}{1-q}\Big)\\- \frac{1-\xi(u)}{(1+z_2)[\xi'(1)-\xi'(q)]}.
\end{align*}
These functions come from substituting the term $\Delta$ in \eqref{gu1} and \eqref{gu2} using \eqref{e10}.

Last, define
\begin{align}
  h_1(u)&=\xi'(u)(q+qz_1+qz_2-uz_1)-(1+z_2)\xi'(q)u, \label{eh1u}\\
  h_2(u)&=[\xi'(u)-\xi'(q)](1+z_2-q-uz_2)-(u-q)(\xi'(1)-\xi'(q)).\label{eh2u}
\end{align}

 We summarize the above calculations in the following Proposition.

\begin{proposition}\label{prop2}
A measure $\nu_{P}$ defined as in \eqref{2rsbga} is the Parisi measure for the model $\xi$ if and only if \eqref{e10} holds, $z_1>0$ and $qz_2>(1-q)z_1$,
\begin{equation}\label{eq:f1f2}
f_{1}(q,z_{2})=f_{2}(q,z_{2}) =0,
\end{equation}
and
\begin{equation}\label{eq:ggggg}
g_{1}(u) \leq 0 \text{ for } u\leq q \text{ and } g_{2}(u) \leq0 \text{ for } u \in [q,1].
\end{equation}
Furthermore, condition \eqref{eq:ggggg} holds if
\begin{align}
h_{1}(u)<0 \text{ in a neighborhood }(0,\delta) \text{ of } 0 \text{ and } h_{1}(u) \text{ has only one zero in } (0,q),  \label{ch1u}\\
 \text{ and } h_{2}(u)<0 \text{ in a neighborhood }(q,q+\delta) \text{ of } q \text{ and } h_{2}(u) \text{ has only one zero in } (q,1). \label{ch2u}
\end{align}

\end{proposition}
\begin{proof}
Suppose $z_2>0$ and $q\in(0,1)$ satisfy \eqref{e1} and \eqref{e2}. If \eqref{e10} holds and $z_1>0$, then all the equations \eqref{e10}--\eqref{e13} are satisfied. The condition $q z_2> (1-q)z_1$ ensures $A_1<A_2$.
The first part of the Proposition follows from the computation done in the last page as conditions \eqref{gu1} and \eqref{gu2} are just $g_1(u)\le 0, u\in[0,q]$ and $g_2(u)\le 0, u\in[q,1]$. We now argue the last claim. By direct computation, we find
\begin{align*}
  g_1'(u)= \frac{q[\xi'(u)(q+qz_1+qz_2-uz_1)-(1+z_2)\xi'(q)u]}{(1+z_2)(1+z_1+z_2) [q(1+z_1+z_2)-uz_1]\xi'(q)} ,
\end{align*}
and
\begin{align*}
  g_2'(u)=\frac{[\xi'(u)-\xi'(q)](1+z_2-q-uz_2)-(u-q)(\xi'(1)-\xi'(q))} {(1+z_2)(1+z_2-q-uz_2)[\xi'(1)-\xi'(q)]}.
\end{align*}
Note that $g_1'(u)$ and $h_1(u)$ have the same sign for $u\in (0,q)$, so do $g_2'(u)$ and $h_2(u)$ for $u\in (q,1)$. By \eqref{e10} and \eqref{eq:f1f2}, $g_1(0)=g_1(q)=g_1'(q)=0$. In order to show $g_1(u)\le 0$ for $u\in [0,q]$, it suffices to show that $h_1(u)<0$ in a neighborhood of $0$ and that $h_1(u)$ has only one zero in $(0,q)$. Similarly, since $g_2(1)=0$, to show $g_2(u)\le 0$ for $u\in [q,1]$, it suffices to show that $h_2(u)<0$ in a neighborhood $(q,q+\de)$ of $q$ and that $h_2(u)$ has only one zero in $(q,1)$.
\end{proof}

\section{Spherical $s+p$ models: a simpler criterion}\label{s:s+p}
Although Proposition \ref{prop2} characterizes a possible 2-RSB phase, it is hard to solve the nonlinear system given by \eqref{eq:f1f2}. Here we try to reduce the difficulty by focusing on $s+p$ models, i.e.
we will take
\[
\xi(x) = (1-\la) x^s+\la x^p
\]
for some $\la\in (0,1), p>s \ge 3$. Since $p+(p+1)$ models are known to be in the 1-RSB phase \cite{AC17}, we assume $p\ge s+2$.

\begin{lemma}\label{l:h1}
  Suppose $z_1>0,z_2>0$ and $q\in(0,1)$ satisfy \eqref{e1} and \eqref{e2}. If $h_1(1)>0$, \eqref{e10} and \eqref{g2u0} hold, then $g_1(u)\le 0$ for $u\in[0,q]$.
\end{lemma}
\begin{proof}
  Note that for the $s+p$ model we consider here, $\frac{h_1(u)}{u}|_{u=0+}<0$. It follows that $h_1(u)<0$ in a neighborhood of $0$. Since $g_1(0)=g_1(q)=0$, $g_1'(u)$ and $h_1(u)$ have at least one zero in $(0,q)$. By Proposition \ref{prop2}, it remains to show that $h_1(u)$ has only one zero in $(0,q)$. Note that
  \[
  h_1(u)= \la p[-z_1 u^p+ q(1+z_1+z_2) u^{p-1}]+(1-\la) s[- z_1u^s + q(1+z_1+z_2) u^{s-1}] - (1+z_2)\xi'(q)u.
  \]
  There are four sign changes for the coefficients of $h_1(u)$. By Descartes' rule of signs, $h_1(u)$ has at most four strictly positive roots counting multiplicity. A calculation yields
  \[
  h_1'(u) = \xi''(u)[q+(q-u)z_1+qz_2]-z_1\xi'(u)-(1+z_2)\xi'(q).
  \]
  The condition \eqref{g2u0} implies that $h_1'(q)\le 0$.
  We also know $h_1(+\8) <0$ and $h_1(1)>0$. Then there is at least one zero greater than 1.

  Since $h_1(q)=0$, if $h_1'(q)<0$, then $h_1(u)$ has a zero in $(q,1)$ and at least three zeros in $[q,+\8)$. Therefore $h_1$ has exactly one zero in $(0,q)$. If $h_1'(q)=0$, then $h_1(u)$ has a root at $q$ with multiplicity at least two. It follows that $h_1(u)$ has at least three roots in $[q,+\8)$ and thus has at most one root in $(0,q)$.
\end{proof}

\begin{lemma}\label{l:h2}
  Suppose $z_2>0$ and $q\in(0,1)$ satisfies \eqref{e1} and \eqref{e2}. If $h_2(0)<0$, \eqref{e10} and \eqref{g2u0} hold, then $g_2(u)\le 0$ for $u\in[q,1]$.
\end{lemma}
\begin{proof}
  Since $g_2(q)=g_2(1)=0$, $g_2'(u)$ and $h_2(u)$ have at least one zero in $(q,1)$. Note that
  \begin{align*}
      h_2(u)= \la p [-z_2 u^p+(1+z_2-q)u^{p-1}]+(1-\la)s[-z_2 u^s +(1+z_2-q)u^{s-1}] \\
      +[\xi'(q)z_2+\xi'(q)-\xi'(1)]u- [\xi'(q)z_2+\xi'(q)-q\xi'(1)].
  \end{align*}
  Thanks to our assumption $h_2(0)<0$, the coefficients of $h_2(u)$ has four sign changes no matter what the linear term is. By Descartes' rule of signs, $h_2(u)$ has at most four zeros in $(0,+\8)$. By calculation,
  \[
  h_2'(u)=\xi''(u)(1+z_2-q-uz_2) -z_2[\xi'(u)-\xi'(q)]-[\xi'(1)-\xi'(q)].
  \]
  It follows from \eqref{g2u0} that $h_2'(q)\le 0$. Note that $h_2(q)=h_2(q')=h_2(1)=0$ for some $q'\in(q,1)$ and $h_2(+\8)<0$.

  Suppose $h_2'(q)<0$. Then $h_2(u)$ has at least one zero in $(0,q)$ and $h_2(u)<0$ in a neighborhood $(q,q+\de)$ of $q$. Therefore, $g_2(u)<0$ in $(q,q+\de)$ and $h_2(u)$ has exactly one zero in $(q,1)$. It follows that $g_2(u)\le 0$ for $u\in[q,1]$. Suppose $h_2'(q)=0$. Then $h_2(u)$ has a root at $q$ with multiplicity at least two. However, we know $h_2(u)$ has at least two roots in $(q,+\8)$. So the multiplicity of $q$ is two and $h_2(u)$ has no zero in $(0,q)$. Then $h_2(u)<0$ in a neighborhood $(q-\de, q+\de)\setminus\{q\}$ of $q$. It follows that $g_2(u)<0$ in $(q,q+\de)$. Since $h_2(u)$ has exactly one zero in $(q,1)$, we conclude $g_2(u)\le 0$ for $u\in[q,1]$.
\end{proof}
\begin{theorem}\label{thm1}
  Let $\xi(x)=(1-\la) x^s +\la x^p$ for $2<s<p-1$, $s,p\in\mathbb{N}$. Then the model $\xi(x)$ is 2-RSB provided the following conditions hold:
  \begin{enumerate}
    \item[(i)] Equation \eqref{1z1z2} holds, $z_1>0,z_2>\frac{(1-q)z_1}q$ and $q\in(0,1)$ satisfy \eqref{e1} and \eqref{e2};
    \item[(ii)] Equation \eqref{g2u0} holds;
    \item[(iii)] $h_1(1)>0$ and $h_2(0)<0$.
  \end{enumerate}
\end{theorem}
\begin{proof}
   By Lemmas \ref{l:h1} and \ref{l:h2}, conditions \eqref{gu1} and \eqref{gu2} are verified. Note that \eqref{e10} and \eqref{1z1z2} are equivalent. Now the assertion follows from the criterion given in Proposition \ref{prop2}.
\end{proof}

Let us now explain how we will check the three conditions in Theorem \ref{thm1}. Here for convenience of writing, we use $q, z_2$ both for a solution of the system \eqref{eq:f1f2} and for the set of solutions of $f_2(q,z_2)=0$. It should be clear from the context what we mean. The general principle we follow here is to show that the implicit function $z_2=\phi(q)$ defined by \eqref{e2} is strictly decreasing under some condition on $\la$, and then use intermediate value property of continuous functions to estimate $q, z_2$ in a small interval.

First, we compute $\frac{d z_2}{ dq}$ with implicit differentiation:
\begin{align*}
\frac{d{f_2(q,z_2)}}{d q} = [-\xi'(1)-\xi'(q)-(1-q)\xi''(q)] \Big( \frac{1+z_2}{z_2^2}\log(1+z_2)-\frac1{z_2}\Big)+\xi''(q)(1-q)\\
-(1-q)[\xi'(1)-\xi'(q)]\Big(\frac{2+z_2}{z_2^3}\log(1+z_2)-\frac2{z_2^2} \Big) \frac{d z_2}{d q}
 =0.
\end{align*}
Solving for $\frac{1+z_2}{z_2^2}\log(1+z_2)-\frac1{z_2}$ in \eqref{e2} and plugging into the above equation, we find
\begin{align*}
 \frac{ d z_2}{d q}= \frac{\xi''(q)(1-q)-[\xi'(1)-\xi'(q) +(1-q)\xi''(q)]\frac{1-\xi(q)-\xi'(q)(1-q)}{(1-q)[\xi'(1)-\xi'(q)]}} {(1-q)[\xi'(1)-\xi'(q)](\frac{2+z_2}{z_2^3}\log(1+z_2)-\frac2{z_2^2})}.
\end{align*}
With the elementary inequality $\log(1+z)\ge \frac{2z}{2+z}$ for $z>0$, we immediately see that the denominator is strictly positive. The numerator is simplified as
\[
\frac{\xi''(q)(1-q)[(1-q)\xi'(1)-1+\xi(q)]-[\xi'(1)-\xi'(q)][1-\xi(q) -\xi'(q)(1-q)]}{(1-q)[\xi'(1)-\xi'(q)]}.
\]
It follows that $\frac{dz_2}{dq} <0$ if and only if
\begin{align}\label{dz2dq}
  \phi_1(q):=\xi''(q)[(1-q)\xi'(1)-1+\xi(q)]-\frac{\xi'(1)-\xi'(q)}{1-q}[1-\xi(q) -\xi'(q)(1-q)]<0.
\end{align}

\begin{proposition}\label{pconla}
  We have $\frac{dz_2}{dq}<0$ provided $\la=0$ or
  \begin{align}\label{conla}
      \la\Big(ps+1-\frac16s^3-p-\frac56s\Big)-\Big(ps+1- \frac16s^3 -\frac16ps^2 -\frac56s-\frac56p\Big)  \ge 0.
  \end{align}
\end{proposition}
We remark that it is not difficult to see that we need to impose some restriction in $\lambda$ in order to obtain  $\frac{dz_2}{dq}<0$. The rest of the section is devoted to the proof of the above proposition.

\begin{proof}
  The argument is a little technical but still elementary. The difficulty here is that we don't have efficient methods to determine the sign of the polynomial in \eqref{dz2dq}. Although not obvious, we will factor out a factor $(q-1)^3$ and then show that the remaining factor has all coefficients positive, which is guaranteed by condition \eqref{conla}.

  We will use the elementary identity $1-q^n=(1-q)\sum_{k=0}^{n-1}q^k$ repeatedly. Since $\xi(1)=1$, we have
  \begin{align*}
    \frac{\phi_1(q)}{1-q}& = \xi''(q)\Big[(1-\la)s+\la p -(1-\la)\sum_{k=0}^{s-1} q^k-\la \sum_{k=0}^{p-1}q^k\Big]
    -\Big[(1-\la)s\sum_{k=0}^{s-2}q^k+\la p\sum_{k=0}^{p-2}q^k \Big] \\
     &\qquad \times \Big[(1-\la)\sum_{k=0}^{s-1} q^k+\la \sum_{k=0}^{p-1}q^k -(1-\la)s q^{s-1}-\la p  q^{p-1}\Big]\\
     &=(1-q)\xi''(q)\Big[ (1-\la)\sum_{k=1}^{s-1}\sum_{j=0}^{k-1}q^{j}+\la \sum_{k=1}^{p-1}\sum_{j=0}^{k-1} q^j \Big]-\Big[(1-\la)s\sum_{k=0}^{s-2}q^k+\la p\sum_{k=0}^{p-2}q^k \Big] \\
     & \qquad \times (1-q) \Big[(1-\la)\sum_{k=0}^{s-2}q^k\sum_{j=0}^{s-k-2}q^j +\la \sum_{k=0}^{p-2}q^k\sum_{j=0}^{p-k-2}q^j \Big].
  \end{align*}
  Let $\phi_2(q)=\frac{\phi_1(q)}{(1-q)^2}$. Using $\xi''(q)=(1-\la)s\sum_{i=0}^{s-2} q^{s-2}+\la p\sum_{i=0}^{p-2} q^{p-2} $, we can rewrite
  \begin{align*}
    \phi_2(q)& = (1-\la)^2s \Big(\sum_{i=0}^{s-2} q^{s-2} \sum_{k=1}^{s-1}\sum_{j=0}^{k-1}q^{j} -\sum_{i=0}^{s-2}q^i \sum_{k=0}^{s-2}q^k\sum_{j=0}^{s-k-2}q^j\Big)\\
    & \ \ +\la^2 p \Big( \sum_{i=0}^{p-2} q^{p-2} \sum_{k=1}^{p-1}\sum_{j=0}^{k-1}q^{j} -\sum_{i=0}^{p-2}q^i \sum_{k=0}^{p-2}q^k\sum_{j=0}^{p-k-2}q^j \Big)\\
    & \ \ + (1-\la)\la s \Big(\sum_{i=0}^{s-2} q^{s-2} \sum_{k=1}^{p-1}\sum_{j=0}^{k-1}q^{j} -\sum_{i=0}^{s-2}q^i \sum_{k=0}^{p-2}q^k\sum_{j=0}^{p-k-2}q^j\Big)\\
    & \ \ +\la (1-\la) p \Big( \sum_{i=0}^{p-2} q^{p-2} \sum_{k=1}^{s-1}\sum_{j=0}^{k-1}q^{j} -\sum_{i=0}^{p-2}q^i \sum_{k=0}^{s-2}q^k\sum_{j=0}^{s-k-2}q^j \Big)\\
    &=: I+II+III+IV.
  \end{align*}
  Let us start with analyzing $III$. For simplicity, we denote the term in the parenthesis by $III'=\frac{III}{(1-\la) \la s}$. By change of variable $k=p-1-\ell$ in the second summand, we have
  \begin{align*}
    III'&=\sum_{i=0}^{s-2} (q^{s-2} -q^i)\sum_{k=1}^{p-1}\sum_{j=0}^{k-1}q^{j} + \sum_{i=0}^{s-2}q^i \sum_{k=1}^{p-1}\sum_{j=0}^{k-1}(q^{j} -q^{j+p-1-k})\\
    &=(q-1)\sum_{i=0}^{s-3} \sum_{k=0}^{s-3-i} q^{i+k}\sum_{k=1}^{p-1}\sum_{j=0}^{k-1}q^{j} + (1-q)\sum_{i=0}^{s-2}q^i \sum_{k=1}^{p-2}\sum_{j=0}^{k-1}q^j\sum_{i=0}^{p-2-k} q^{i}.
  \end{align*}
  Now let $III''=III'/(q-1)$ and we will find a closed formula for $III''$. By expanding all the sums and recombining, we find
  \begin{align*}
    III'' &= \sum_{k=1}^{s-2} k q^{k-1} \sum_{j=1}^{p-1} (p-j)q^{j-1} - \sum_{i=1}^{s-1}q^{i-1} \sum_{j=1}^{p-2} j(p-1-j)q^{j-1}\\
    &=\sum_{n=p-2}^{p+s-5} q^n\sum_{j=1}^{p+s-4-n} (n-p+2+j)j - \sum_{n=p-2}^{p+s-5} q^n\sum_{j=1}^{p+s-4-n} (p-1-j)j\\
    &\quad +\sum_{n=s-2}^{p-3} q^n\sum_{j=1}^{s-2} (p-2-n+j)j - \sum_{n=s-2}^{p-3} q^n\sum_{j=1}^{s-1} (n+2-j)(p-n-3+j)\\
    &\quad +\sum_{n=0}^{s-3} q^n \sum_{j=1}^{n+1} (p-2-n+j)j - \sum_{n=0}^{s-3} q^n \sum_{j=1}^{n+1} (p-1-j)j
  \end{align*}
  Using the elementary summation formulae for $\sum_1^m j$ and $\sum_1^m j^2$, we have
  \begin{align*}
    III'' &= \sum_{n=p-2}^{p+s-5} q^n III''_1  +\sum_{n=s-2}^{p-3}q^n III''_2+\sum_{n=0}^{s-3}q^n \frac16(n+1)(n+2)(n+3),
  \end{align*}
  where
  \begin{align*}
  III''_1&=\frac16(p+s-n-3)(p+s-n-4)(-2p+4s-n-5)\\
    III''_2&= \sum_{j=1}^{s-2}[2j^2+(2p-3n-7)j -(n+2)(p-n-3)] -(n+3-s)(p+s-n-4)\\
    &=(p-\frac32n-\frac72)(s-1)(s-2)+(s-2)(s-1)(\frac23s-1)\\
    &\quad -(n+2)(p-3-n)(s-2) -(n+3-s)(p+s-4-n)\\
    &= (s-1)n^2-[ps- p+\frac32 s^2-\frac{15}2s+6] n +(p-\frac72)(s-1)(s-2) \\ &\quad +(s-2)(s-1)(\frac23ps-1)-2(s-2)(p-3)+(p+s-4)(s-3).
  \end{align*}
  Note that the coefficient of $n$ is negative, which implies that $III''_2$ is an increasing function of $n$ for $n\in \mathbb{N}$. We see also that the coefficients of $q^n$ for $n\le s-3$ are all positive.

  Let us now turn to the quantity $IV$ and define $IV'=\frac{IV}{(1-\la)\la p}$. By the same argument as for $III'$, we find a factor $q-1$ in $IV'$ and define $IV''=IV'/(q-1)$. Then we have
  \begin{align*}
    IV''&=\sum_{k=1}^{p-2} k q^{k-1} \sum_{j=1}^{s-1} (s-j)q^{j-1} - \sum_{i=1}^{p-1}q^{i-1} \sum_{j=1}^{s-2} j(s-1-j)q^{j-1}\\
    &=\sum_{n=p-2}^{p+s-5} q^n\sum_{j=1}^{p+s-4-n} (n-s+j+2)j -\sum_{n=p-2}^{p+s-5}q^n \sum_{j=1}^{p+s-4-n} (s-1-j)j \\
    &\quad +\sum_{n=s-2}^{p-3} q^n\sum_{j=1}^{s-1} (n+2-s+j)j -\sum_{n=s-2}^{p-3} q^n\sum_{j=1}^{s-2} (s-1-j) j\\
    &\quad +\sum_{n=0}^{s-3} q^n\sum_{j=1}^{n+1} (s-2-n+j) j -\sum_{n=0}^{s-3} q^n \sum_{j=1}^{n+1} (s-1-j) j\\
    % &=\sum_{n=p-2}^{p+s-5} q^n\sum_{j=1}^{p+s-4-n} (n-2s+2j+3)j \\ &+\sum_{n=s-2}^{p-3} q^n \sum_{j=1}^{s-2} (n-2s+2j+3) j +\sum_{n=0}^{s-3} q^n\sum_{j=1}^{n+1}(2j-1-n)j\\
    &=\sum_{n=p-2}^{p+s-5} q^n IV''_1 + \sum_{n=s-2}^{p-3} q^n IV''_2+\sum_{n=0}^{s-3}q^n \frac16(n+1)(n+2)(n+3)
  \end{align*}
  where
  \begin{align*}
    IV''_1&=\frac16(p+s-n-3)(p+s-n-4)(4p-2s-n-5)\\
    IV''_2&=\frac16s(s-1)(3n-2s+7).
  \end{align*}
  Similarly, we define $I''=\frac{I}{(1-\la)^2s(q-1)}$ and $II''=\frac{II}{\la^2p(q-1)}$. Then we find
  \begin{align*}
    I''&=\sum_{n=s-2}^{2s-5}q^n\sum_{s=1}^{2s-4-n} (n-s+2+j)j +\sum_{n=0}^{s-3} q^n \sum_{j=1}^{n+1} (s+j-n-2) j\\
    &\quad -\sum_{n=s-2}^{2s-5}q^n\sum_{s=1}^{2s-4-n} (s-1-j)j -\sum_{n=0}^{s-3} q^n \sum_{j=1}^{n+1} (s-1-j)j\\
    &=\sum_{n=s-2}^{2s-5}q^n\sum_{j=1}^{2s-4-n}(n-2s+2j+3)j +\sum_{n=0}^{s-3} q^n\sum_{j=1}^{n+1}(2j-n-1)j\\
    &=\sum_{n=s-2}^{2s-5}q^n\frac16(2s-n-3)(2s-n-4)(2s-n-5) +\sum_{n=0}^{s-3}q^n\frac16(n+1)(n+2)(n+3)\\
    &\ge 0
  \end{align*}
 and
\begin{align*}
  II''&=\sum_{n=p-2}^{2p-5}q^n\sum_{j=1}^{2p-4-n}(n-2p+2j+3)j +\sum_{n=0}^{p-3} q^n\sum_{j=1}^{n+1}(2j-n-1)j\\
  &=\sum_{n=p-2}^{2p-5}q^n \frac16(2p-n-3)(2p-n-4)(2p-n-5) +\sum_{n=0}^{p-3}q^n\frac16(n+1)(n+2)(n+3)\\
  &\ge0.
\end{align*}
Both are positive because all their coefficients are positive. We will drop the positive term $I''$ and determine the sign of $\frac1{q-1}(II +III +IV )$. Note that for $n\le s-3$ or $n\ge p+s-4$, the coefficients of $\frac1{q-1}(II +III +IV )$ are all positive. For $p-2\le n\le p+s-5$, we have
\begin{align*}
  s III_1''+ p IV''_1&=\frac16(p+s-n-3)(p+s-n-4)[3(p-s)^2+(p+s)(p+s-5-n)]>0
\end{align*}
as $p+s-5-n\ge0$ and $p-s>0$. So the coefficients are strictly positive for $p-2\le n\le p+s-5$. It remains to determine the coefficients for $s-2\le n\le p-3$. Let $II''_2=\frac16(n+1)(n+2)(n+3)$ be the coefficient of $q^n$ in $II''$. Note that $II''_2,III''_2$ and $IV''_2$ are all increasing as $n$ increases. It suffices to consider the coefficient of $q^{s-2}$ in $\frac1{q-1}(II +III +IV )$, which is
\begin{align*}
  [\la^2p II''_2+(1-\la)\la s III''_2+(1-\la)\la p IV''_2 ]|_{n=s-2}\\
  =s\la\Big[\frac16s^3+\frac16ps^2-ps+\frac56(s+p)-1+\la (ps-p-\frac16s^3 -\frac56s +1)\Big].
\end{align*}
This quantity is no less than zero exactly as \eqref{conla} holds. Thus $\frac{dz_2}{dq}$ is strictly negative because we have strict positivity for the coefficients when $ p-2 \le n\le p+s-5$. This ends the proof of Proposition~\ref{pconla}.
\end{proof}

%\begin{remark}
%  Some kind of conditions on $\la$ are needed for $\frac{dz_2}{dq}<0$. For instance, one can check for the model $s=4,p=760, \la=1/30$, $\frac{dz_2}{dq}>0$ for some $q\in(0,1)$.
%\end{remark}

\section{Examples of 2-RSB models and the proof of Theorem \ref{mainthm}}\label{sec:examples}
We will use Theorem \ref{thm1} and Proposition \ref{pconla} to determine models that are in the 2-RSB phase. The method is as follows: we first use the intermediate value property for the continuous functions $f_1$ and $f_2$ defined in \eqref{e1} and \eqref{e2}, and condition \eqref{conla} to estimate $q,z_2$ in a small interval, and then check the conditions in Theorem \ref{thm1}. The zero set of $f_1$ is more involved, but to guarantee a solution for \eqref{e1} and \eqref{e2} we only need a mild property of this zero set, which we now explain.

Let $[q^-,q^+]\subset(0,1)$ and assume that $f_1(q^-,\cdot)$ (resp.~$f_1(q^+,\cdot)$ has a zero $z_2^1$ (resp.~$z_2^2$) in a closed interval $[\hat z^-(q^-),\hat z^+(q^-)]$ (resp.~$[\hat z^-(q^+),\hat z^+(q^+)])$ contained in $(0,+\8)$. Recall  $z_2^s$ and $z_2^b$ as given in \eqref{e:z2s} and \eqref{e:z2b} and that $z_{2}^{b}(\cdot)$ is continuous on $(0,1)$. If we know
\begin{align}\label{z2scheck}
z_2^s(q^-) > \hat z^+(q^-) \ \ \text{and} \ \ z_2^s(q^+)> \hat z^+(q^+),
\end{align}
then $z_2^1$ and $z_2^2$ are both the smaller zeros, i.e. we have $z_2^1<z_2^b(q^-)$ and $z_2^2<z_2^b(q^+)$. For $q_0\in[q^-,q^+]$, let $\zeta(q_0)$ denote the smaller zero of $f_1(q_0,\cdot)$ as in Proposition \ref{p0f1f2}. Then $f_1(q_0,\zeta(q_0))=0$, $\partial_{z_2}f_1(q_0, \zeta(q_0))<0$. By the implicit function theorem, there exists a unique continuous function $z_2=\tilde\zeta(q)$ defined in a neighborhood $U$ of $q_0$ such that $\zeta(q_0)=\tilde\zeta(q_0)$ and $f_1(q,\tilde\zeta(q))=0$ for $q\in U$. We claim that the function $\tilde \zeta$ has to coincide with $\zeta$ on $U$.
Indeed, suppose that
\[
q^{*}:=\inf \{q \geq q_{0} : q\in U, \tilde \zeta (q)= z_{2}^{b}(q)\}<\8.
\]
By continuity of the function $z_{2}^{b}(\cdot)$, $z_{2}^{b}(q^{*})=0$. Since $\zeta(q^{*})< z^{b}_{2}(q^{*})$, another application of the implicit function theorem, now at $(q^{*}, \zeta(q^{*}))$, shows the existence of $\epsilon >0$ so that $\zeta(q^{*}-\epsilon)< \tilde \zeta(q^{*}-\epsilon)$. Thus $\tilde \zeta(q^{*}-\epsilon)=z_{2}^{b}(q^{*}-\epsilon),$ proving that $q^{*}$ does not exist. This shows that
\[
\zeta(q)=\tilde \zeta(q)\  \text{for} \ q\ge q_0, q\in U.
\]
Similarly, we can show that $\zeta(q)=\tilde \zeta(q)$ {for}  $q\le q_0, q\in U$.
Since the argument above is valid at every point $q_{0}$ of the interval $[q^{-},q^{+}]$, compactness implies that there is a continuous curve given by the graph of $z_2=\zeta(q)$ which connects the points $(q^-,z_2^1)$ and $(q^+,z_2^2)$ if condition \eqref{z2scheck} is verified.

Suppose we know a solution of \eqref{e1} and \eqref{e2} satisfies $q\in[q^-,q^+], z_2\in [z_2^-,z_2^+]$ for some $0~<~q^{-}<q^{+}<1$ and $z_{2}^{-}, z_{2}^{+}>0$. Then from \eqref{1z1z2} we know
\begin{align}\label{z1est}
  \frac{q^-[\xi'(1)-\xi'(q^+)]}{\xi'(q^+)(1-q^-)}-z_2^+-1>0
\end{align}
 ensures $z_1>0$, and
\begin{align}\label{z12est}
  \frac{q^+[\xi'(1)-\xi'(q^-)]}{\xi'(q^-)(1-q^+)}-z_2^- -1 < \frac{q^- z_2^-}{1-q^-}
\end{align}
 ensures $z_2>\frac{z_1(1-q)}{q}$. Since \eqref{1z1z2} is used to compute $z_1$, both equations of \eqref{g2u0} will be fulfilled provided
\begin{align}\label{g2qest}
  \xi''(q^+)(1+z_2^+)(1-q^-)+\xi'(q^+)-\xi'(1)<0.
\end{align}
Substituting $z_{1}$ by \eqref{1z1z2}, we know
\begin{align*}
  h_1(1)=(1+z_2)[\xi'(1)q-\xi'(q)+\xi'(1)(1-q)] -\frac{\xi'(1)q[\xi'(1)-\xi'(q)]}{\xi'(q)}.
\end{align*}
Clearly, as $q$ increases, the first term is decreasing while the second is increasing. So in order to check $h_1(1)>0$ it suffices to check
\begin{align}\label{h11est}
  (1+z_2^-)[\xi'(1)q^+-\xi'(q^+)+\xi'(1)(1-q^+)] -\frac{\xi'(1)q^-[\xi'(1)-\xi'(q^-)]}{\xi'(q^-)}>0.
\end{align}
Since $h_2(0)= \xi'(1)q-\xi'(q)(1+z_2)$, if we assume the condition
\begin{align}\label{h20est}
  \xi'(1)q^+-(1+z_2^-)\xi'(q^-)<0,
\end{align}
it follows that $h_2(0)<0$.

Before going to the examples, we define
\begin{align}\label{aba}
  G=\log\frac{\xi''(1)}{\xi'(1)}-\frac{[\xi''(1)-\xi'(1)][\xi''(1)-\xi'(1) +\xi'(1)^2]} {\xi''(1)\xi'(1)^2}
\end{align}
as in \cite[Definition 4.1]{AB}. The model $\xi$ is called pure-like if $G>0$, full mixture if $G<0$, and critical if $G=0$. Based on their study of critical points of the Hamiltonian $H_{N}$, Auffinger and Ben Arous asked whether the 1-RSB phase coincides with pure-like models (Question 4.1 in \cite{AB}). In pure-like models the average number of local minima near the ground state is exponentially larger than the number of saddles of positive index while in full mixture models these appear in the same number at exponential scale. Jagannath and Tobasco \cite{JT16} provided a negative answer to this question by showing that some pure-like $2+p$ models are not 1-RSB. The choice of the $2$ component is special in this counter-example as the spherical SK model does not have positive complexity\cite[Remark 2.3]{ABC}\footnote{The model $\xi(x)=x^{2}$ is trivial. The Hamiltonian is just a quadratic form and the critical points are simply the eigenvectors of a $N \times N$ GOE matrix. Positive complexity means that the number of critical points is of order $e^{cN}$ for some $c>0$.}. The question was still open if the model has no 2-spin component.

\begin{remark}\label{rem111}
Our examples below show that among the $s+p, s\ge3,$ 2-RSB models, there exist both pure-like and full mixture models. Therefore, this provides a strong negative answer to the question of Auffinger and Ben Arous, and shows that in general the level of RSB cannot be classified by the critical points consideration.
\end{remark}

The next examples provide the proof of Theorem \ref{mainthm}.

\begin{example}[A pure-like $2$-RSB model]\rm Consider
  \[
  \xi(x)=\frac57 x^3 + \frac27 x^{16}.
  \]
  Since $G>0$, this model is pure-like.
  By Proposition \ref{pconla}, a calculation shows that $\frac{dz_2}{dq}<0$ provided $\la>\frac7{39}$; and here we have $\la=\frac27>\frac7{39}$. By numerical calculations, we obtain the following:
  \begin{align*}
     f_1(0.743,3.2)>0,\ f_1(0.743,3.22)<0,\ f_1(0.747,3.2)>0,\ f_1(0.747,3.22)<0,\\
     f_2(0.743,3.22)>0,\ f_2(0.743,3.25)<0,\ f_2(0.747,3.17)>0,\ f_2(0.747,3.2)<0,
  \end{align*}
and also that \eqref{z2scheck} holds.
 From the second line, when $q=0.743$, $f_2(0.743,z_2)$ has a zero in $(3.22, 3.25)$; and when $q=0.747$, $f_2(0.747,z_2)$ has a zero in $(3.17, 3.2)$. Since $z_2$ is strictly decreasing as $q$ increases by Proposition \ref{pconla}, the zero set of $f_2(q,z_2)$ satisfies that if $q\in(0.743, 0.747)$, $z_2\in(3.17,3.25)$. To see there is a solution for \eqref{e1} and \eqref{e2}, the first line shows that $f_1(0.743,z_2)$ has a zero $z_2^1$ in $(3.2,3.22)$ and that $f_1(0.747,z_2)$ has a zero $z_2^2$ in $(3.2,3.22)$. Since \eqref{z2scheck} holds, there is a continuous curve connecting the two points $(0.743, z_2^1)$ and $(0.747, z_2^2)$ as explained above. By continuity there is at least one common element in $(q,z_2)\in[0.743,0.747]\times[3.17,3.25]$ in the zero sets of $f_1(q,z_2)$ and $f_2(q,z_2)$. Choosing $q^-=0.743,q^+=0.747,z_2^-=3.17, z_2^+=3.25$, the conditions \eqref{z1est}, \eqref{z12est}, \eqref{g2qest}, \eqref{h11est} and \eqref{h20est} are verified by numerical calculations. By Theorem \ref{thm1}, any solution $(q,z_2)\in[0.743,0.747]\times[3.17,3.25]$ gives a 2-RSB Parisi measure. But the Parisi measure is unique. Therefore, there is in fact exactly one solution in $(q,z_2)\in[0.743,0.747]\times[3.17,3.25]$.

\end{example}
\begin{remark}
The $3+16$ model was predicted by physicists to be 2-RSB for some choice of coefficients and some positive temperature; see \cite[Figure 2]{CL07}. However, it is not clear from the prediction in \cite{CL07} if the zero temperature case had the same behavior. For instance, by \cite{ACZ17} all mixed $p$-spin (Ising) models are FRSB at zero temperature, while this is in general not expected to be true throughout the low temperature regime.
\end{remark}
\begin{example}[A full mixture $2$-RSB model]\rm For our second example we take
\[
\xi(x)=\frac56x^3+\frac16x^{16}.
\]
Since $G<0$, this is a full mixture model. By numerical calculations, we have
\begin{align*}
  f_1(0.824, 1.58)>0,\ f_1(0.824, 1.6)<0,\
f_1(0.828, 1.57)>0,\ f_1(0.828, 1.6)<0,\\
f_2(0.824, 1.6)>0,\ f_2(0.824, 1.64)<0,\ f_2(0.828, 1.54)>0,\ f_2(0.828, 1.57)<0,
\end{align*}
and condition  \eqref{z2scheck} holds.
From here we see that there is a solution to \eqref{e1} and \eqref{e2} in $(q,z_2)\in[0.824,0.828]\times [1.54, 1.64]$. If we take $q^-=0.824,q^+=0.828, z_2^-=1.54, z_2^+=1.64$, then the conditions  \eqref{z1est}, \eqref{z12est}, \eqref{g2qest}, \eqref{h11est} and \eqref{h20est} are verified. By Theorem \ref{thm1}, the model belongs to the 2-RSB phase.
\end{example}

The following examples show that 2-RSB exists even for convex functions $\xi$. Convexity of $\xi$ was used before to simplify many questions in spin glasses.

\begin{example}\rm
$\xi(x)=\frac56 x^4+\frac16 x^{40}$. Since $G<0$, this model is full mixture. By proposition \ref{pconla}, $\frac{dz_2}{dq}<0$ if $\la>7/107$, which is satisfied here. Numerical calculations yield that
\begin{align*}
  f_1(0.89, 3.6)>0,\ f_1(0.89, 3.8)<0, \ f_1(0.9, 3.7)>0, \ f_1(0.9, 3.9)<0,\\
f_2(0.89, 3.9)>0,\ f_2(0.89, 4.1)<0,\ f_2(0.9, 3.5)>0, \ f_2(0.9, 3.7)<0,
\end{align*}
and condition \eqref{z2scheck} holds.
By the same analysis as above, there is a solution to \eqref{e1} and \eqref{e2} in $(q,z_2)\in[0.89,0.9]\times[3.5,4.1]$. We take $q^-=0.89, q^+=0.9, z_2^-=3.5, z_2^+=4.1$. Then  \eqref{z1est}, \eqref{z12est}, \eqref{g2qest}, \eqref{h11est} and \eqref{h20est} are verified. So $\xi$ is a 2-RSB model.
\end{example}

\begin{example}\rm
$\xi(x)=\frac58x^4+\frac38x^{40}$. Since $G>0$, this model is pure-like. We have $\frac{dz_2}{dq}<0$ as $\frac38>\frac7{107}$. By calculation, we find
\begin{align*}
  f_1(0.83, 9.2)>0,\ f_1(0.83, 9.51)<0,\ f_1(0.833, 9.4)>0,\ f_1(0.833, 9.8)<0,\\
f_2(0.83, 9.51)>0,\ f_2(0.83, 9.8)<0,\ f_2(0.833, 9.3)>0,\ f_2(0.833, 9.4)<0,
\end{align*}
and condition \eqref{z2scheck} holds.
For the same reason as before, we take $q^-=0.83, q^+=0.833, z_2^-=9.3, z_2^+=9.8$ and the conditions  \eqref{z1est}, \eqref{z12est}, \eqref{g2qest}, \eqref{h11est} and \eqref{h20est} are verified. So it is a $2$-RSB model.
\end{example}

\begin{remark}
It would be nice to derive exact conditions on the coefficient $\lambda$ of $\xi(x)=(1-\lambda) x^{s} + \la x^{p}$ that guarantee $2$-RSB. However, this seems to be difficult and not even known in the $1$-RSB case.
\end{remark}

\begin{remark}
Let us stress that our argument here is very specific to the 2-RSB phase  and to $s+p$ models in many aspects. The difficulty goes beyond perseverance. For instance, the four sign change phenomenon for $h_1$ and $h_2$ observed in Lemmas \ref{l:h1} and \ref{l:h2} does not happen if we consider a model with three mixed spins. Proposition \ref{pconla} is tricky to observe, but it is essential in order to bypass the properties of the zero set of $f_1$ as defined in \eqref{e1}, which is hard to study analytically. In summary, our argument is a combination of various non-trivial, specific observations for the $s+p$ models and for the 2-RSB phase. In fact, many miraculous coincidences occur so that our argument works. %we believe we have pushed the available techniques to the limit for studying $k$-RSB phases beyond $k=1$.
\end{remark}

\section{Proof of Theorem \ref{thm:2RSB} and Corollary \ref{cor:1}}\label{proofofthm2}

There is an alternative formulation of the Parisi formula for the maximum energy proved in \cite{AC17}. For $\nu\in\kx$, define
\[
\hat \nu(s)=\int_s^1 \xi''(r)\nu(dr), \ \ s\in[0,1].
\]
Let $\ux=\{(B,\nu)\in \rz\times \kx: \hat\nu(0)<B\}$ and define the Parisi functional $\ux$ by
\[
\px(B,\nu)= \frac12\Big(\int_0^1 \frac{\xi''(s) ds}{B-\hat\nu(s)} +B - \int_0^1 s\xi''(s)\nu(ds)\Big).
\]
Then the Parisi formula \cite[Theorem 10]{AC17} states that
\[
\mathcal Q(\nu_P) = \inf_{(B,\nu)\in \ux} \px(B,\nu),
\]
and the the infimum is achieved by a unique $(B_P,\nu_P)\in\ux$. In this case
\begin{align}\label{bnu}
B_P=\hat\nu_P(0)+\frac1{\nu_P([0,1])}.
\end{align}
Let $\ga$ be a 2-RSB Parisi measure given as in \eqref{2rsbga}. Using \eqref{bnu} and dropping the subscript, we have
\[
B = A_1\xi'(q)+A_2[\xi'(1)-\xi'(q)]+\Delta \xi''(1)+\frac1{A_1q +A_2(1-q)+\Delta}
\]
and
\[
B-\hat\nu(s) = \begin{cases}
A_1 \xi'(s) +\frac1{A_1q +A_2(1-q)+\Delta}, & s\le q,\\
A_1 \xi'(q)+ A_2 (\xi'(s)-\xi'(q)) + \frac1{A_1q +A_2(1-q)+\Delta}, & q\le s \le 1.
\end{cases}
\]

For $a\in [0,q]$ and $b\in [q,1]$, let us define
\begin{align}
\psi_1(a)&= \frac{(1+z_1+z_2)q \xi'(a)}{z_1} - a\xi'(a) +\xi(a)\notag \\
&\quad  -\frac{q\xi'(q)(1+z_2)(1+z_1+z_2)}{z_1^2} \log\Big(1+\frac{z_1\xi'(a)}{(1+z_2)\xi'(q)}\Big), \label{eh1a}\\
\psi_2(b)& = \xi'(b)(q-b)+\xi(b)-\xi(q)+\frac{[\xi'(b)-\xi'(q)](1-q)}{z_2 } +[\xi'(b)-\xi'(q)](1-q) \notag\\
& \quad -\frac{(1-q)(1+z_2)[\xi'(1)-\xi'(q)]}{z_2^2}\log \Big(1+\frac{[\xi'(b)-\xi'(q)]z_2}{\xi'(1)-\xi'(q)}\Big).\label{eh2b}
\end{align}

%Theorem \ref{thm:2RSB} now follows from the next proposition.
%\begin{proposition}
%Let $\xi$ be a 2-RSB convex model. Suppose $\psi_1(a)<0$ for all $a\in(0,q)$ and $\psi_2(b)<0$ for all $b\in(q,1)$. Then for any $\eps>0$ there exist $\eta,K>0$ such that for all $N\ge 1$,
%\begin{align*}
%\mathbb{P}_N(\eta, [-1+\eps, -q-\eps]\cup [-q+\eps, -\eps]\cup [\eps, q-\eps]\cup [q+\eps,1-\eps]) \le K e^{-N/K}.
%\end{align*}
%
%\end{proposition}

\begin{proof}[Proof of Theorem \ref{thm:2RSB}]
Let $\xi$ be a 2-RSB convex model with Parisi measure given by $\gamma$ as above. Suppose $\psi_1(a)<0$ for all $a\in(0,q)$ and $\psi_2(b)<0$ for all $b\in(q,1)$. We will show that for any $\eps>0$ there exist $\eta,K>0$ such that for all $N\ge 1$,
\begin{align}\label{eq:wwwtp}
\mathbb{P}_N(\eta, [-1+\eps, -q-\eps]\cup [-q+\eps, -\eps]\cup [\eps, q-\eps]\cup [q+\eps,1-\eps]) \le K e^{-N/K}.
\end{align}

The idea of the proof is to use the Guerra-Talagrand inequality for the coupled minima energy (\cite[Theorem 11]{AC17}). In our context, it says that if we define for a Borel subset $A\subset [-1,1]$,
	\begin{align*}
	M_N(A)&:=\frac{1}{N}\e \min_{R_{1,2}\in A}\bigl(H_N(\sigma^1)+H_N(\sigma^2)\bigr),
	\end{align*}
and if
\[
\px_a(B,\nu_m)=\int_0^1 \frac{\xi''(s)ds}{B-\hat\nu_m(s)} +B - \int_0^1 s\xi''(s) \nu_m(ds),
\]	
where $0\le a\le q$,  $\frac12 \le m\le 2$,
\[
\nu_m(ds)=  (A_1 m\indi_{[0,a)}(s)+A_1 \indi_{[a,q)}(s)+A_2\indi_{[q,1)}(s)  )ds+\Delta \de_{\{1\}}(ds),
\]
with $\hat\nu_m(s) = \int_s^1 \xi''(r)\nu_m(dr)$.
Then for $u\in[-q,q]$ and $a=|u|$,
\begin{align}\label{lem5:eq1}
		\lim_{\varepsilon\downarrow 0}\liminf_{N\rightarrow\infty} M_{N} \bigl((u-\varepsilon,u+\varepsilon)\bigr)&\geq -\mathcal{P}_{a}(B,\nu_{1}).
		\end{align}
Observe that $\px_a(B,\nu_1)=2\px(B,\nu_P)=2\qx(\nu_P)$.

We denote $|\nu|=\nu[0,1] =A_1q +A_2(1-q)+\Delta$.	 Then we have
\[
B-\hat\nu_m(s)=
\begin{cases}
  A_1\xi'(a)+A_1 [\xi'(s)-\xi'(a)]m +\frac1{|\nu|},& 0\le s\le a,\\
  A_1\xi'(s)+\frac1{|\nu|},& a\le s\le q,\\
  A_1\xi'(q)+A_2[\xi'(s)-\xi'(q)]+\frac1{|\nu|},& q\le s\le 1.
\end{cases}
\]
By dominated convergence theorem, we may differentiate inside the integral and get
\begin{align*}
  \partial_m \px_a(B,\nu_m) = -\int_0^a \frac{A_1[\xi'(s)-\xi'(a)]d\xi'(s)}{(A_1\xi'(a) +A_1[\xi'(s)-\xi'(a)]m +\frac1{|\nu|})^2}-A_1\int_0^a s\xi''(s) ds.
\end{align*}
Plugging in $m=1$ and evaluating the integral, we find
\begin{align*}
  \frac{\partial_m \px_a(B,\nu_m)|_{m=1}}{A_1} &= -\int_0^a \frac{[\xi'(s)-\xi'(a)]d\xi'(s)}{(A_1\xi'(s)+\frac1{|\nu|})^2} - a\xi'(a) +\xi(a)\\
  &=-\frac1{A_1^2}\log(1+A_1\xi'(a)|\nu|)+ \frac{\xi'(a)|\nu|}{A_1} - a\xi'(a) +\xi(a).
\end{align*}
Now using the change of variables \eqref{e10} and comparing with \eqref{eh1a}, we have
\begin{align*}
  \frac{\partial_m \px_a(B,\nu_m)|_{m=1}}{A_1} &=\frac{(1+z_1+z_2)q \xi'(a)}{z_1} - a\xi'(a) +\xi(a) \notag \\
 &\quad -\frac{q\xi'(q)(1+z_2)(1+z_1+z_2)}{z_1^2} \log\Big(1+\frac{z_1\xi'(a)}{(1+z_2)\xi'(q)}\Big) \notag \\
 &=\psi_1(a) <0, \quad \forall a\in(0,q).
\end{align*}

Similarly, for $q\le b \le 1$ and $\frac12\le m\le 2$, we let
\[
\nu_m(ds) = (A_1 \indi_{[0,q)}(s)+A_2 m \indi_{[q,b)}(s) +A_2\indi_{[b,1)}(s)) ds +\Delta \de_{\{1\}}(ds).
\]
Then we have
\[
B-\hat \nu_m(s) = \begin{cases}
A_1\xi'(s) +A_2[\xi'(q)-\xi'(b)]m + A_2[\xi'(b)-\xi'(q)] +\frac1{|\nu|}, & 0\le s \le q,\\
A_1\xi'(q)+A_2[\xi'(b)-\xi'(q)]+A_2[\xi'(s)-\xi'(b)]m +\frac1{|\nu|}, & q\le s\le b,\\
A_1\xi'(q) + A_2[\xi'(s)-\xi'(q)]+\frac1{|\nu|}, & b\le s\le 1.
\end{cases}
\]
Let
\[
\px_b(B,\nu_m) = \int_0^1 \frac{\xi''(s) ds}{B-\hat\nu_m(s)} +B-\int_0^1 s\xi''(s) \nu_m(ds).
\]
It follows that $\px_b(B,\nu_1) = 2\px(B,\nu_P) = 2\qx(\nu_P)$. Differentiating inside the integral, we find
\[
\partial_m\px_b(B,\nu_m) = -A_2\int_0^q \frac{[\xi'(q)-\xi'(b)] d\xi'(s)}{[B-\hat\nu_m(s)]^2}-A_2\int_q^b \frac{[\xi'(s)-\xi'(b)]d\xi'(s)}{[B-\hat\nu_m(s)]^2} - A_2\int_q^b s\xi''(s) ds.
\]
Plugging in $m=1$, we have
\begin{align*}
\frac{\partial_m\px_b(B,\nu_m) |_{m=1}}{A_2} &= -\int_0^q\frac{[\xi'(q)-\xi'(b)]d\xi'(s)}{(A_1\xi'(s)+\frac1{|\nu|})^2} -\int_q^b \frac{[\xi'(s)-\xi'(b)]d\xi'(s)}{(A_1\xi'(q)+A_2[\xi'(s)-\xi'(q)]+\frac1{|\nu|})^2}\\
&\quad - b\xi'(b)+q\xi'(q)+\xi(b)-\xi(q)\\
&= I + II - b\xi'(b)+q\xi'(q)+\xi(b)-\xi(q).
\end{align*}
Evaluating the integrals yields
\begin{align*}
I = \frac1{A_1}[\xi'(q)-\xi'(b)]\frac1{A_1\xi'(s)+\frac1{|\nu|}}\Big|_{s=0}^q = \frac{[\xi'(b)-\xi'(q)]|\nu| \xi'(q)}{ A_1 \xi'(q)+\frac1{|\nu|}},
\end{align*}
and
\begin{align*}
II &= -\frac1{A_2} \int_q^b \frac{d\xi'(s)}{A_1\xi'(q)+A_2[\xi'(s)-\xi'(q)]+\frac1{|\nu|}} \\
&\quad + \frac{A_1\xi'(q)  +A_2[\xi'(b)-\xi'(q)]+\frac1{|\nu|}}{A_2}\int_q^b \frac{d\xi'(s)}{(A_1\xi'(q)+A_2[\xi'(s)-\xi'(q)]+\frac1{|\nu|})^2}\\
&=-\frac1{A_2^2} \log\Big(1+\frac{A_2[\xi'(b)-\xi'(q)]}{A_1\xi'(q)+\frac1{|\nu|}}\Big)+\frac{\xi'(b)-\xi'(q)}{A_2(A_1\xi'(q)+\frac1{|\nu|})}.
\end{align*}
Using the change of variables \eqref{e10} and comparing with \eqref{eh2b}, we can rewrite
\begin{align*}
\frac{\partial_m\px_b(B,\nu_m) |_{m=1}}{A_2}& = \xi'(b)(q-b)+\xi(b)-\xi(q)+\frac{[\xi'(b)-\xi'(q)](1-q)}{z_2 } +[\xi'(b)-\xi'(q)](1-q)\\
& \ -\frac{(1-q)(1+z_2)[\xi'(1)-\xi'(q)]}{z_2^2}\log \Big(1+\frac{[\xi'(b)-\xi'(q)]z_2}{\xi'(1)-\xi'(q)}\Big)\\
&= \psi_2(b)<0, \quad \forall b\in(q,1).
\end{align*}
The rest of the proof follows the same argument as for \cite[Theorem 6]{AC17}. We reproduce it here for completeness and clarity for the general audience.

Since $(a,m)\mapsto\partial_m\mathcal{P}_a(B,\nu_m)$ is continuous on $[\eps,q-\eps]\times [\frac12,2]$ and on $[q+\eps,1-\eps]\times [\frac12,2]$, from the mean value theorem, there exist $m$ around $1$ and $\eta>0$ such that for any $u$ with $|u|\in[\varepsilon, q-\varepsilon] \cup [q+\varepsilon,1-\varepsilon]$,
		\begin{align*}
		\mathcal{P}_{|u|} (B,\nu_m)&\leq \mathcal{P}_{|u|} (B,\nu_1)-4\eta\\
		&=2\qx(\nu_P) -4\eta.
		\end{align*}
		Therefore, from \eqref{lem5:eq1} and an analogous inequality with lower bound $-\px_b(B,\nu_1)$ for $|u|=b\in [q, 1]$, we find for any $u$ satisfying $|u|\in[\varepsilon, q-\varepsilon] \cup [q+\varepsilon,1-\varepsilon]$,
		\begin{align*}
		\lim_{\varepsilon'\downarrow 0}\liminf_{N\rightarrow\infty}M_{N} \bigl((u-\varepsilon',u+\varepsilon')\bigr)>-2\qx(\nu_P) +4\eta.
		\end{align*}
We now proceed with a covering argument. Since $\mathcal S:=[-1+\varepsilon,-q-\varepsilon] \cup [-q+\varepsilon, -\varepsilon] \cup [\varepsilon, q-\varepsilon] \cup [q +\varepsilon, 1-\varepsilon]$ is compact, we can cover it with a finite collection of intervals of the type $(a_{i}-\epsilon_{i}, a_{i} +\epsilon_{i})$, $i =1, \ldots, n$ for some $a_{i} \in \mathcal S$. Therefore from the above display, there exists $N_0\ge 1$ such that
\[
M_{N} \bigl(\mathcal S \bigr)>-2\qx(\nu_P) +3\eta,
\]
for $N\ge N_0$.
Next, a standard argument from concentration of measure for Gaussian extrema processes implies that there exists $K>0$ such that with probability at least $1-Ke^{-N/K}$,

\begin{align}\label{thm2:proof:eq-2}
\frac{1}{N}\min_{R_{1,2}\in \mathcal S}\bigl(H_N(\sigma^1)+H_N(\sigma^2)\bigr)\geq -2\qx(\nu_P) + 2\eta.
\end{align}
If there exist $\sigma^1,\sigma^2$ such that $R_{1,2}\in \mathcal S$, $H_N(\sigma^1)\leq N(-\qx(\nu_P)+\eta/2)$, and $H_N(\sigma^2)\leq N(-\qx(\nu_P)+\eta/2)$, then
		
\begin{align*}
\frac{H_N(\sigma^1)+H_N(\sigma^2)}{N}\leq  -2\qx(\nu_P)+ \eta.
\end{align*}
From \eqref{thm2:proof:eq-2}, this means that $\p_N(\eta/2,\mathcal S)\leq Ke^{-N/K}$ for all $N\geq N_0$ and this clearly implies \eqref{eq:wwwtp}, ending the proof of the theorem.
\end{proof}
\begin{remark}
Note that Examples 3 and 4 considered in the previous section satisfy the assumptions (and conclusion) of Theorem 2.
\end{remark}
\begin{remark}\label{rem3}
Note that $\psi_1(0)=0$ and $\psi_1(q)=0$ by \eqref{e10} and \eqref{e11}. Moreover, we know
\begin{align*}
  \psi_1'(a) &=\xi''(a)\Big[\frac{(1+z_1+z_2)q}{z_1}-a-\frac{q\xi'(q)(1+z_2) (1+z_1+z_2)}{z_1[(1+z_2)\xi'(q)+z_1\xi'(a)]} \Big]\\
  & = \frac{\xi''(a)[\xi'(a)[(1+z_1+z_2)q-z_1a] -\xi'(q)(1+z_2)a]}{(1+z_2)\xi'(q)+z_1\xi'(a)}.
\end{align*}
Comparing $\psi_1'(a)$ with the function $h_1(u)$ as in \eqref{eh1u}, we see that $\psi_1'(a)$ and $h_1(a)$ have the same sign for $0\le a\le q$. Similarly, note that $\psi_2(q)=0$ and $\psi_2(1)=0$ by \eqref{e10} and \eqref{e12}. Moreover, a calculation gives
\begin{align*}
\psi_2'(b) &= \xi''(b)q-b\xi''(b)+\frac{\xi''(b)(1-q)}{z_2}+\xi''(b)(1-q)-\frac{\xi''(b)(1-q)(1+z_2)[\xi'(1)-\xi'(q)]}{z_2[\xi'(1)-\xi'(q)+(\xi'(b)-\xi'(q))z_2]}\\
&=\frac{ \xi''(b)[(q-b)[\xi'(1)-\xi'(q)]+[\xi'(b)-\xi'(q)](1-q+(1-b)z_2 ) ] }{\xi'(1)-\xi'(q)+[\xi'(b)-\xi'(q)]z_2}.
\end{align*}
Comparing with $h_2(u)$ as in \eqref{eh2u}, we see that $\psi_2'(b)$ and $h_2(b)$ have the same sign for $q\le b \le 1$.

This means that $\psi_1(a)<0, a\in(0,q)$ and $\psi_2(b)<0,b\in(q,1)$ will hold once the conditions \eqref{ch1u} and \eqref{ch2u} are enforced.
\end{remark}

\begin{proof}[Proof of Corollary \ref{cor:1}]
For any $\sigma \in \mathcal L(\eta)$ let $\mathcal C_{x}$ denote the connected component of $x$ in $\mathcal L(\eta)$, that is,
\[ \mathcal C_{x}=\bigg \{ y \in \mathcal L(\eta): \exists \gamma:[0,1] \to S_{N} \text{ continuous  with } \gamma(0)=x, \gamma(1)=y, \gamma(t) \in \mathcal L(\eta), \; \forall t \in [0,1] \bigg\}.
\]

Let $x, y \in \mathcal L(\eta)$ so that $|R_{12}(x,y)| \leq \epsilon$. By choosing $\epsilon$ small enough, any path between $x$ and $y$ must contain two points that have overlap in $(S_{\varepsilon}\cup -S_{\varepsilon})^{c}$. From Theorem \ref{thm:2RSB}  this implies that with high probability
\[
\mathcal C_{x} \cap \mathcal C_{y} = \emptyset,
\]
where $\eta$ is taken from Theorem \ref{thm:2RSB}. Since $H_{N}$ is a smooth function almost surely, each of this disjoint connected components must contain at least a local minimum. Now choose $\eta'$ and $\epsilon'$ so that $0<\eta'+\epsilon' < \eta$. Note that $\mathcal L(\eta') \subseteq \mathcal L(\eta)$. The Corollary now follows directly by equation \eqref{es}, since the random set $O_{N}$ guarantees the existence of exponentially many of such points in $\mathcal L(\eta')$.
\end{proof}

We end the paper with one extra remark about the computations above. This may be of independent interest.
\begin{remark}
The functions $h_1$ and $h_2$ are essential for the 2-RSB phase. Indeed, we have seen in Remark \ref{rem3} that the derivatives of $\psi_1$ and $g_1$ (resp. $\psi_2$ and $g_2$) have the same sign as $h_1$ (resp. $h_2$) in their domain. In fact, $h_1$ and $h_2$ also appear in another situation.

There is an alternative formulation of the minimizer condition  for the Parisi functional in \cite[Proposition 3]{AC17}. Let us recall here. Let
\[
\bar f(s) = \int_s^1 f(r)\xi''(r) dr \quad \text{where} \quad f(r) =\int_0^r \frac{\xi''(t) dt}{(B-\hat \nu(s))^2}-r.
\]
Let $\ga$ be a 2-RSB Parisi measure given as in \eqref{2rsbga}. By direct computation, we have for $0\le s\le q$,
\begin{align*}
\bar f(s) &=-\xi'(1) + s\xi'(s)+1-\xi(s) +\frac{[\xi'(q)-\xi'(s)][A_1q + A_2 (1-q)+\Delta]}{A_1} \\
&+ \frac{\xi'(1)-\xi'(q)}{A_2(A_1\xi'(q)+\frac1{A_1 q +A_2(1-q) +\Delta})} -\frac1{A_1^2}\log\frac{A_1\xi'(q)(A_1q+A_2(1-q)+\Delta)+1}{A_1\xi'(s)(A_1q +A_2(1-q)+\Delta)+1}\\
&-\frac1{A_2^2}\log\Big(1+ \frac{A_2[\xi'(1)-\xi'(q)]}{A_1\xi'(q)+\frac1{A_1q+A_2(1-q)+\Delta}}\Big)\\
&+\frac{\xi'(1)-\xi'(q)}{A_1}\Big( A_1q+A_2(1-q)
+\Delta-\frac{A_1q+A_2(1-q)+\Delta}{A_1\xi'(q)(A_1q+A_2(1-q)+\Delta)+1}\Big),
\end{align*}
and for $q\le s\le 1$,
\begin{align*}
\bar f(s) &= -\xi'(1) + s\xi'(s)+1-\xi(s) +\frac{\xi'(1)-\xi'(s)}{A_2}\frac{A_1q+A_2(1-q)+\Delta}{A_1\xi'(q)(A_1q+A_2(1-q)+\Delta)+1}\\
&+\frac{\xi'(1)-\xi'(s)}{A_1}\Big( A_1q+A_2(1-q)
+\Delta-\frac{A_1q+A_2(1-q)+\Delta}{A_1\xi'(q)(A_1q+A_2(1-q)+\Delta)+1}\Big)\\
&-\frac1{A_2^2}\log\frac{A_1\xi'(q)+A_2(\xi'(1)-\xi'(q))+\frac1{A_1q+A_2(1-q)+\Delta}}{A_1\xi'(q)+A_2(\xi'(s)-\xi'(q))+\frac1{A_1q+A_2(1-q)+\Delta}}.
\end{align*}
Using the change of variables \eqref{e10}, we rewrite for $0\le s\le q$,
\begin{align*}
\bar f(s) &= -\xi'(1) + s\xi'(s)+1-\xi(s) +\frac{[\xi'(q)-\xi'(s)]q(1+z_1+z_2)}{z_1}+\frac{[\xi'(1)-\xi'(q)](1-q)(1+z_2)}{z_2}\notag \\
&\qquad -\frac{q\xi'(q)(1+z_2)(1+z_1+z_2)}{z_1^2} \log \frac{(1+z_1+z_2)\xi'(q)}{(1+z_2)\xi'(q)+z_1\xi'(s)}\notag\\
&\qquad -\frac{(1-q)(1+z_2)[\xi'(1)-\xi'(q)]}{z_2^2}\log(1+z_2)+[\xi'(1)-\xi'(q)]q \notag\\
&= s\xi'(s)+\xi(q)-\xi(s)-q\xi'(q)  +\frac{q[\xi'(q)-\xi'(s)](1+z_1+z_2)}{z_1} \notag\\
&\qquad -\frac{q\xi'(q)(1+z_2)(1+z_1+z_2)}{z_1^2} \log \frac{(1+z_1+z_2)\xi'(q)}{(1+z_2)\xi'(q)+z_1\xi'(s)},
\end{align*}
where in the second equation we used \eqref{e12}. Similarly for $q\le s\le 1$, we have
\begin{align*}
  \bar f(s) &=  s\xi'(s)-\xi'(s) +1-\xi(s) +\frac{[\xi'(1)-\xi'(s)](1-q)}{z_2} \notag \\
  &\ -\frac{(1-q)(1+z_2)(\xi'(1)-\xi'(q))}{z_2^2} \log\frac{1+z_2}{1+\frac{z_2(\xi'(s)-\xi'(q))}{\xi'(1)-\xi'(q)}}.
\end{align*}
By \cite{AC17}, $(B,\nu)\in\ux$ is the minimizer of $\px$ if and only if
\begin{align}
f(1)=0, \bar f(q)=0, \bar f(0) = 0,  \label{fff}\\
\bar f(u)\ge0 \text{ for } u\in[0,1].\label{barf0}
\end{align}
In this case, we necessarily have $f(q)=0$. One can check that the conditions \eqref{fff} hold using \eqref{e10}--\eqref{e13}. The expression of $\bar f(s)$, however, looks very different than $g(u)$ as in \eqref{gu1} and \eqref{gu2}. Nevertheless, using \eqref{e10}--\eqref{e12} one can directly check that $\bar f'(s)$ and $-h_1(s)$ have the same sign for $s\in[0,q]$ and that $\bar f'(s)$ and $-h_2(s)$ have the same sign for $s\in[q,1]$.

This means the conditions \eqref{ch1u} and \eqref{ch2u} given in Proposition \ref{prop2} are sufficient for $\bar f(s)\ge 0$ as in \eqref{barf0}, as well as $\psi_1(u)\le 0, \psi_2(u)\le 0, g_1(u)\le 0, g_2(u)\le 0$.

\end{remark}

\bibliographystyle{plain}
\bibliography{spinglass}
\end{document}